\documentclass[11pt]{article}
\usepackage[T1]{fontenc}
\usepackage[utf8]{inputenc}
\usepackage{lmodern}
\usepackage{microtype}
\usepackage[margin=1in]{geometry}
\usepackage{amsmath,amssymb,amsthm,mathtools}
\usepackage{enumitem}
\usepackage{xcolor}
\usepackage[colorlinks=true,linkcolor=blue!50!black,citecolor=blue!50!black,urlcolor=blue!50!black]{hyperref}
\usepackage[nameinlink,capitalise]{cleveref}
\numberwithin{equation}{section}

\newtheorem{theorem}{Theorem}[section]
\newtheorem{proposition}[theorem]{Proposition}

\newtheorem{corollary}[theorem]{Corollary}
\theoremstyle{remark}
\newtheorem{remark}[theorem]{Remark}
\theoremstyle{definition}

\newcommand{\E}{\mathbb{E}}

\newcommand{\1}{\mathbf{1}}
\newcommand{\OO}{\mathcal{O}}
\newcommand{\eps}{\varepsilon}
\newcommand{\diag}{\mathrm{diag}}
\newcommand{\ind}{\mathrm{ind}}
\newcommand{\simu}{\mathrm{sim}}
\newcommand{\parlay}{\mathrm{par}}
\newcommand{\sing}{\mathrm{sing}}

\title{Optimal Parlay Wagering and Whitrow Asymptotics\\
\large A State-Price and Implicit-Cash Treatment}
\author{Christopher D. Long}
\date{}

\begin{document}
\maketitle

\begin{abstract}
For independent multi-outcome events under multiplicative parlay pricing, we give a short exact proof of the optimal Kelly strategy using the implicit-cash viewpoint. The proof is entirely eventwise. One first solves each event in isolation. The full simultaneous optimizer over the entire menu of singles, doubles, triples, and higher parlays is then obtained by taking the outer product of the one-event Kelly strategies. Equivalently, the optimal terminal wealth factorizes across events. This yields an immediate active-leg criterion: a parlay is active if and only if each of its legs is active in the corresponding one-event problem. The result recovers, in a more transparent state-price form, the log-utility equivalence between simultaneous multibetting and sequential Kelly betting.

We then study what is lost when one forbids parlays and allows only singles. In a low-edge regime and on a fixed active support, the exact parlay optimizer supplies the natural reference point. The singles-only problem is a first-order truncation of the factorized wealth formula. A perturbative expansion shows that the growth-rate loss from forbidding parlays is $\OO(\eps^4)$, while the optimal singles stakes deviate from the isolated one-event Kelly stakes only at cubic order. This yields a clean explanation of Whitrow's empirical near-proportionality phenomenon: the simultaneous singles-only optimizer is obtained from the isolated eventwise optimizer by an event-specific cubic shrinkage, so the portfolios agree through second order and differ only by a small blockwise drag.
\end{abstract}

\section{Introduction}
Kelly's expected-log criterion is explicit for a single multi-outcome event: after sorting outcomes by the edge ratio $p_i/\pi_i$, one obtains a threshold support and a closed-form stake formula. In the language of state prices, cash is an implicit claim on every outcome, and explicit bets merely top up those outcomes whose ratio $p_i/\pi_i$ exceeds the cash floor. This viewpoint is developed in \cite{LongSingle2026}.

For several independent events played simultaneously, the picture depends strongly on the betting menu. If only singles are allowed, the terminal-wealth profile is additive across events, and the weights remain coupled through the logarithm; see, for example, Whitrow \cite{Whitrow2007} and the support-decoupling analysis in \cite{LongUtility2026}. If one is also allowed to bet parlays under multiplicative pricing, the attainable wealth profiles become much richer. Grant, Johnstone, and Kwon \cite{GrantJohnstoneKwon2008} proved that for log utility the simultaneous multibet problem is monetarily equivalent to sequential Kelly betting.

The first purpose of this note is to show that the exact multibet optimizer has a one-line structural description in the implicit-cash framework. Solve each event separately. If the one-event optimizer on event $\ell$ has cash $c_\ell^*$ and stakes $s_{\ell i}^*$, then the optimal stake on the ticket that selects outcomes $i_\ell$ on a subset $T$ of events is simply
\[
  x_{\gamma}^* = \prod_{\ell\in T} s_{\ell,i_\ell}^* \prod_{r\notin T} c_r^*.
\]
Thus the entire ticket book is the outer product of the one-event strategies. The proof uses only the one-event KKT conditions and independence.

The second purpose is to make precise what is lost by restricting the bettor to singles. Once the exact parlay formula is known, the singles-only problem can be viewed as the first-order truncation of an exact product wealth profile. That observation turns Whitrow's numerical phenomenon into a perturbative theorem. On a fixed active support and in a low-edge regime, the optimal singles-only portfolio differs from the isolated one-event Kelly portfolio only by an event-specific cubic shrinkage. The value loss from excluding parlays is smaller still: it begins at quartic order.

\section{Single-event Kelly in implicit-cash form}
Consider a single event with outcomes $i=1,\dots,n$, bettor probabilities $p_i>0$ with $\sum_i p_i=1$, and state prices $\pi_i>0$. A stake vector $s=(s_1,\dots,s_n)$ and cash level $c\ge 0$ satisfy
\begin{equation}
  c+\sum_{i=1}^n s_i = 1,
  \label{eq:single-budget}
\end{equation}
and produce terminal wealth
\begin{equation}
  W_i = c + \frac{s_i}{\pi_i}
  \qquad (i=1,\dots,n).
  \label{eq:single-wealth}
\end{equation}
The single-event Kelly problem is
\begin{equation}
  \max_{c,s}\; \sum_{i=1}^n p_i \log W_i
  \quad \text{subject to \eqref{eq:single-budget}, } c\ge 0,\ s_i\ge 0.
  \label{eq:single-kelly}
\end{equation}

\begin{proposition}[Single-event implicit-cash formula]
\label{prop:single-event}
Assume the optimal single-event solution has positive cash. Then there is a unique $c^*\in(0,1)$ such that
\begin{equation}
  s_i^* = (p_i-c^*\pi_i)_+,
  \qquad
  W_i^* = c^* + \frac{s_i^*}{\pi_i} = \max\!\left(c^*,\frac{p_i}{\pi_i}\right).
  \label{eq:single-formula}
\end{equation}
Moreover,
\begin{equation}
  \E\!\left[\frac{1}{W^*(I)}\right] = 1,
  \qquad
  \frac{p_i}{\pi_i W_i^*}
  \begin{cases}
    =1,& s_i^*>0,\\
    <1,& s_i^*=0.
  \end{cases}
  \label{eq:single-kkt-clean}
\end{equation}
\end{proposition}

\begin{proof}
Introduce a multiplier $\lambda$ for \eqref{eq:single-budget}. Since $c^*>0$, the KKT conditions are
\[
  \frac{p_i}{\pi_i W_i}\le \lambda,
  \qquad
  s_i\!\left(\lambda-\frac{p_i}{\pi_i W_i}\right)=0,
  \qquad
  \sum_{i=1}^n \frac{p_i}{W_i} = \lambda.
\]
Let $A=\{i:s_i>0\}$, and write $P_A:=\sum_{i\in A}p_i$ and $Q_A:=\sum_{i\in A}\pi_i$. For $i\in A$, complementary slackness gives
\[
  W_i=\frac{p_i}{\lambda\pi_i}.
\]
Substituting this into the stationarity condition with respect to $c$ yields
\[
  \lambda
  =\sum_{i\in A}\frac{p_i}{W_i}+\sum_{i\notin A}\frac{p_i}{c}
  =\lambda Q_A+\frac{1-P_A}{c},
\]
hence
\begin{equation}
  1-P_A=\lambda c(1-Q_A).
  \label{eq:single-PA-QA-1}
\end{equation}
Also, for $i\in A$,
\[
  s_i=\pi_i(W_i-c)=\frac{p_i}{\lambda}-c\pi_i.
\]
Summing over $A$ and using the budget identity gives
\[
  1-c=\sum_{i\in A}s_i=\frac{P_A}{\lambda}-cQ_A.
\]
Multiplying by $\lambda$,
\begin{equation}
  \lambda-P_A=\lambda c(1-Q_A).
  \label{eq:single-PA-QA-2}
\end{equation}
Comparing \eqref{eq:single-PA-QA-1} and \eqref{eq:single-PA-QA-2} shows $\lambda=1$. Therefore
\[
  W_i=\frac{p_i}{\pi_i}\quad (i\in A),
\]
and for $i\notin A$ one has $W_i=c$, so the inactive KKT inequalities are exactly
\[
  \frac{p_i}{\pi_i c}\le 1.
\]
Equivalently,
\[
  s_i=\pi_i\!\left(\max\!\left(c,\frac{p_i}{\pi_i}\right)-c\right)
  =(p_i-c\pi_i)_+,
\]
and
\[
  W_i=\max\!\left(c,\frac{p_i}{\pi_i}\right).
\]
This proves \eqref{eq:single-formula} and \eqref{eq:single-kkt-clean}. Uniqueness follows from strict concavity.
\end{proof}

\begin{remark}
If one sorts the outcomes so that $p_i/\pi_i$ decreases with $i$, then the active set is a prefix and the cash level is $c^*=(1-P_A)/(1-Q_A)$, where $P_A=\sum_{i\in A}p_i$ and $Q_A=\sum_{i\in A}\pi_i$; see \cite{LongSingle2026}. In typical fixed-odds settings the overround condition $\sum_i \pi_i>1$ guarantees positive cash and at least one inactive outcome.
\end{remark}

\section{The full ticket menu and the exact parlay formula}
Now consider $m\ge 1$ independent events. Event $\ell$ has outcomes $\Omega_\ell=\{1,\dots,n_\ell\}$, probabilities $p_{\ell i}$, and state prices $\pi_{\ell i}$. A ticket is a vector
\[
  \gamma=(\gamma_1,\dots,\gamma_m),
  \qquad \gamma_\ell\in\{0\}\cup \Omega_\ell,
\]
where $\gamma_\ell=0$ means that the ticket omits event $\ell$. Let $\Gamma$ be the set of all such tickets. For $\gamma\in\Gamma$, write
\[
  \pi_\gamma := \prod_{\ell:\,\gamma_\ell\neq 0} \pi_{\ell,\gamma_\ell},
\]
with the empty product equal to $1$. A stake vector $x=(x_\gamma)_{\gamma\in\Gamma}$ satisfies
\begin{equation}
  x_\gamma\ge 0,
  \qquad
  \sum_{\gamma\in\Gamma} x_\gamma = 1.
  \label{eq:ticket-budget}
\end{equation}
If the realized outcome vector is $I=(I_1,\dots,I_m)$, the ticket $\gamma$ pays $x_\gamma/\pi_\gamma$ exactly when it matches every selected leg, that is, when $\gamma_\ell=I_\ell$ for all $\ell$ with $\gamma_\ell\neq 0$. The terminal wealth is therefore
\begin{equation}
  W_x(I)=\sum_{\gamma\preceq I} \frac{x_\gamma}{\pi_\gamma},
  \label{eq:ticket-wealth}
\end{equation}
where $\gamma\preceq I$ denotes compatibility in this sense. The simultaneous multibet Kelly problem is
\begin{equation}
  V_\parlay := \sup_{x}\; \E[\log W_x(I)]
  \quad \text{subject to \eqref{eq:ticket-budget}.}
  \label{eq:parlay-problem}
\end{equation}

For each event $\ell$, let $(c_\ell^*,s_{\ell i}^*)$ be the unique one-event optimizer from \cref{prop:single-event}, and define the one-event optimal wealth factor
\begin{equation}
  F_\ell(i):=c_\ell^*+\frac{s_{\ell i}^*}{\pi_{\ell i}}.
  \label{eq:event-factor}
\end{equation}

\begin{theorem}[Exact optimal parlay formula]
\label{thm:parlay-main}
Assume each one-event Kelly problem has positive cash. Define
\begin{equation}
  x_\gamma^* := \prod_{\ell:\,\gamma_\ell=0} c_\ell^*
  \prod_{\ell:\,\gamma_\ell\neq 0} s_{\ell,\gamma_\ell}^*.
  \label{eq:ticket-product-formula}
\end{equation}
Then $x^*$ is optimal for \eqref{eq:parlay-problem}. Its terminal wealth factorizes as
\begin{equation}
  W_{x^*}(I)=\prod_{\ell=1}^m F_\ell(I_\ell),
  \label{eq:parlay-factorization}
\end{equation}
and therefore
\begin{equation}
  V_\parlay = \sum_{\ell=1}^m \E[\log F_\ell(I_\ell)].
  \label{eq:parlay-value-factorization}
\end{equation}
Equivalently, the optimal full ticket book is the outer product of the isolated one-event Kelly strategies.
\end{theorem}

\begin{proof}
The budget is immediate from repeated expansion:
\[
  \sum_{\gamma\in\Gamma} x_\gamma^*
  = \prod_{\ell=1}^m \left(c_\ell^*+\sum_{i=1}^{n_\ell} s_{\ell i}^*\right)=1.
\]
Likewise, evaluating \eqref{eq:ticket-wealth} at $x^*$ and grouping the compatible tickets event by event gives
\[
  W_{x^*}(I)
  =\prod_{\ell=1}^m\left(c_\ell^*+\frac{s_{\ell,I_\ell}^*}{\pi_{\ell,I_\ell}}\right)
  =\prod_{\ell=1}^m F_\ell(I_\ell),
\]
which is \eqref{eq:parlay-factorization}. By independence,
\[
  \E[\log W_{x^*}(I)]
  =\sum_{\ell=1}^m \E[\log F_\ell(I_\ell)].
\]
Thus it remains only to verify optimality.

Let
\[
  L(x,\lambda)=\E[\log W_x(I)]-\lambda\left(\sum_{\gamma\in\Gamma}x_\gamma-1\right).
\]
For any ticket $\gamma$,
\begin{equation}
  \frac{\partial L}{\partial x_\gamma}(x,\lambda)
  =\E\!\left[\frac{\1_{\{\gamma\preceq I\}}}{\pi_\gamma W_x(I)}\right]-\lambda.
  \label{eq:ticket-gradient}
\end{equation}
At the candidate $x^*$, \eqref{eq:parlay-factorization} and independence yield
\begin{align}
  \E\!\left[\frac{\1_{\{\gamma\preceq I\}}}{\pi_\gamma W_{x^*}(I)}\right]
   &= \prod_{\ell:\,\gamma_\ell\neq 0}
      \frac{p_{\ell,\gamma_\ell}}{\pi_{\ell,\gamma_\ell}F_\ell(\gamma_\ell)}
      \prod_{\ell:\,\gamma_\ell=0}
      \E\!\left[\frac{1}{F_\ell(I_\ell)}\right].
      \label{eq:ticket-factor-grad}
\end{align}
By \cref{prop:single-event}, the second product equals $1$, and for each selected leg
\[
  \frac{p_{\ell i}}{\pi_{\ell i}F_\ell(i)}
  \begin{cases}
    =1,& s_{\ell i}^*>0,\\
    <1,& s_{\ell i}^*=0.
  \end{cases}
\]
Hence \eqref{eq:ticket-factor-grad} is exactly $1$ for every ticket in the support of $x^*$ and is strictly smaller than $1$ for every ticket containing at least one inactive leg. Choosing $\lambda=1$, all KKT conditions hold. Since the feasible set is a simplex and the objective is concave in $x$, $x^*$ is optimal.
\end{proof}

\begin{corollary}[Active-ticket criterion]
\label{cor:active-ticket}
A ticket is active in the optimal multibet book if and only if every selected leg is active in the corresponding one-event Kelly problem. In particular, no optimal parlay contains a singly inactive leg.
\end{corollary}

\begin{proof}
Under \eqref{eq:ticket-product-formula}, $x_\gamma^*>0$ exactly when every factor associated with a selected leg is positive, that is, exactly when every selected one-event stake $s_{\ell,\gamma_\ell}^*$ is positive.
\end{proof}

\begin{corollary}[Explicit ticket formulas]
\label{cor:explicit-tickets}
Let $T\subseteq\{1,\dots,m\}$ and choose an outcome $i_\ell\in\Omega_\ell$ for each $\ell\in T$. The optimal stake on the ticket that selects precisely those legs is
\begin{equation}
  x_T^*(i_\ell:\ell\in T)
  = \prod_{\ell\in T} s_{\ell,i_\ell}^* \prod_{r\notin T} c_r^*.
  \label{eq:general-ticket}
\end{equation}
In particular, the full $m$-leg parlay on $(i_1,\dots,i_m)$ has stake
\begin{equation}
  x_{i_1,\dots,i_m}^* = \prod_{\ell=1}^m s_{\ell,i_\ell}^*.
  \label{eq:full-ticket}
\end{equation}
\end{corollary}

\begin{remark}[Two 1X2 matches]
If two soccer matches have isolated one-event Kelly strategies
\[
  (c_1;h_1,d_1,a_1),\qquad (c_2;h_2,d_2,a_2),
\]
then the optimal simultaneous multibet book is the outer product of these two vectors. Thus the full two-leg parlays are $HH=h_1h_2$, $HD=h_1d_2$, $HA=h_1a_2$, $DH=d_1h_2$, and so on, while the single on home in match $1$ only has stake $h_1c_2$, the single on away in match $2$ only has stake $c_1a_2$, and the pure cash ticket has stake $c_1c_2$.
\end{remark}

\section{The singles-only restriction as a truncated product}
The exact parlay theorem reveals what the singles-only restriction removes. Fix an active support family. After restricting each event to its active outcomes, we work in reduced coordinates: $x_j$ denotes the vector of active stakes on event $j$, and $Z_j^{(\eps)}$ denotes the corresponding active-support excess-return vector. Thus $x_j^\top Z_j^{(\eps)}$ is the one-event top-up on block $j$, with $\eps$ a small edge parameter. On that fixed support, the exact parlay wealth is
\begin{equation}
  W^\parlay(x)=\prod_{j=1}^m \bigl(1+x_j^\top Z_j^{(\eps)}\bigr),
  \label{eq:Wpar}
\end{equation}
whereas the singles-only wealth is the first-order truncation
\begin{equation}
  W^\sing(x)=1+\sum_{j=1}^m x_j^\top Z_j^{(\eps)}.
  \label{eq:Wsing}
\end{equation}
Thus parlays contribute exactly the missing interaction terms
\[
  \sum_{j<k} (x_j^\top Z_j^{(\eps)})(x_k^\top Z_k^{(\eps)})
  + \cdots + \prod_{j=1}^m x_j^\top Z_j^{(\eps)}.
\]
The perturbative question is how strongly the optimizer reacts when those terms are removed.

For Sections~4--6 we impose the following standing low-edge assumptions. For each active block $j$, let
\begin{equation}
  \mu_j(\eps):=\E[Z_j^{(\eps)}]=a_j\eps+\OO(\eps^2),
  \qquad
  M_j(\eps):=\E\bigl[Z_j^{(\eps)}Z_j^{(\eps)\top}\bigr]=C_j+\OO(\eps),
  \label{eq:mu-C}
\end{equation}
where $C_j$ is positive definite. Assume also that the event blocks are independent, the returns are uniformly bounded, and the chosen active support remains fixed for sufficiently small $\eps$. Here $M_j(\eps)$ is the second-moment matrix, not the covariance matrix; since $\mu_j(\eps)=\OO(\eps)$, the distinction is immaterial at the orders used below. Define the isolated and singles-only local objectives
\begin{align}
  \mathcal G_{j,\eps}^{\ind}(y)
    &:= \E\Bigl[\log\bigl(1+y^\top Z_j^{(\eps)}\bigr)\Bigr],
    \label{eq:Gind}\\
  \mathcal G_{\eps}^{\sing}(x_1,\dots,x_m)
    &:= \E\Bigl[\log\Bigl(1+\sum_{j=1}^m x_j^\top Z_j^{(\eps)}\Bigr)\Bigr].
    \label{eq:Gsim}
\end{align}
Let $x_j^{\ind}(\eps)$ be the isolated maximizer and $x^{\simu}(\eps)=(x_1^{\simu}(\eps),\dots,x_m^{\simu}(\eps))$ the singles-only simultaneous maximizer.

\begin{proposition}[First-order isolated ray]
\label{prop:first-order-ray}
With $\alpha_j:=C_j^{-1}a_j$, one has
\begin{equation}
  x_j^{\ind}(\eps)=\alpha_j\eps+\OO(\eps^2).
  \label{eq:first-order-ray}
\end{equation}
\end{proposition}

\begin{proof}
The isolated score map is
\[
  g_j(y,\eps):=\E\!\left[\frac{Z_j^{(\eps)}}{1+y^\top Z_j^{(\eps)}}\right].
\]
At $(y,\eps)=(0,0)$ one has $g_j(0,0)=0$ and
\[
  D_y g_j(0,0)=-C_j,
\]
which is invertible. The implicit function theorem yields a unique $C^1$ stationary branch through the origin. Differentiating $g_j(x_j^{\ind}(\eps),\eps)=0$ at $\eps=0$ gives
\[
  -C_j\dot x_j^{\ind}(0)+a_j=0,
\]
so $\dot x_j^{\ind}(0)=C_j^{-1}a_j=\alpha_j$.
\end{proof}

The next result is the cubic shrinkage formula.

\begin{theorem}[Whitrow asymptotics on a fixed support]
\label{thm:whitrow}
Set
\begin{equation}
  \Lambda_j:=\sum_{k\ne j} a_k^\top C_k^{-1}a_k.
  \label{eq:Lambda}
\end{equation}
Then, for each active block $j$,
\begin{equation}
  x_j^{\simu}(\eps)
  =x_j^{\ind}(\eps)-\Lambda_j\alpha_j\eps^3+\OO(\eps^4)
  =\bigl(1-\Lambda_j\eps^2\bigr)x_j^{\ind}(\eps)+\OO(\eps^4).
  \label{eq:whitrow-main}
\end{equation}
Consequently, the simultaneous singles-only optimizer agrees with the isolated eventwise Kelly optimizer through second order, and the first portfolio interaction is a cubic, event-specific, blockwise scalar shrinkage.
\end{theorem}

\begin{proof}
Write
\[
  u_j:=x_j^\top Z_j^{(\eps)},
  \qquad
  T_j:=\sum_{k\ne j} x_k^\top Z_k^{(\eps)},
  \qquad
  U:=u_j+T_j.
\]
The singles-only score map for block $j$ is
\[
  F_j(x,\eps):=\E\!\left[\frac{Z_j^{(\eps)}}{1+\sum_{\ell=1}^m x_\ell^\top Z_\ell^{(\eps)}}\right].
\]
Because each block stake is $\OO(\eps)$, one has $U=\OO(\eps)$ uniformly, so
\[
  \frac{1}{1+U}=1-U+U^2-U^3+\OO(\eps^4).
\]
Substituting this into $F_j$ and using independence of $T_j$ and $Z_j^{(\eps)}$ yields the decomposition
\begin{equation}
  F_j(x,\eps)
  = g_j(x_j,\eps)
    + \bar T_j\bigl(-\mu_j+2M_jx_j\bigr)
    + V_j\bigl(\mu_j-3M_jx_j\bigr)
    + \OO(\eps^4),
  \label{eq:Fj-decomp}
\end{equation}
where
\[
  \mu_j:=\E[Z_j^{(\eps)}],
  \qquad
  M_j:=\E\bigl[Z_j^{(\eps)}Z_j^{(\eps)\top}\bigr],
  \qquad
  \bar T_j:=\E[T_j],
  \qquad
  V_j:=\E[T_j^2].
\]
The only mixed cubic contributions come from the $-U^3$ term in the geometric expansion. The potentially dangerous ones are of the forms
\[
  \E[Z_j u_j^2]\bar T_j,
  \qquad
  \E[Z_j]\bar T_j^2,
  \qquad
  \E[Z_j u_j]V_j,
  \qquad
  \E[Z_jT_j^3],
\]
and each is $\OO(\eps^4)$ because $u_j=\OO(\eps)$, $\bar T_j=\OO(\eps^2)$, $V_j=\OO(\eps^2)$, and $\E[Z_j]=\OO(\eps)$.

Now evaluate \eqref{eq:Fj-decomp} at the isolated vector $x^{\ind}(\eps)$. Since $g_j(x_j^{\ind}(\eps),\eps)=0$, only the cross-event terms remain. By \cref{prop:first-order-ray},
\[
  x_k^{\ind}(\eps)=\alpha_k\eps+\OO(\eps^2),
  \qquad
  \mu_k=a_k\eps+\OO(\eps^2).
\]
Hence
\begin{align}
  \bar T_j
   &= \sum_{k\ne j}\E\bigl[x_k^{\ind}(\eps)^\top Z_k^{(\eps)}\bigr]
    = \sum_{k\ne j}\alpha_k^\top a_k\,\eps^2+\OO(\eps^3)
    = \Lambda_j\eps^2+\OO(\eps^3),
    \label{eq:Tbar-asymp}\\
  V_j
   &= \sum_{k\ne j}\E\bigl[(x_k^{\ind}(\eps)^\top Z_k^{(\eps)})^2\bigr]+\OO(\eps^4)
    = \sum_{k\ne j}\alpha_k^\top C_k\alpha_k\,\eps^2+\OO(\eps^3)
    = \Lambda_j\eps^2+\OO(\eps^3),
    \label{eq:Vj-asymp}
\end{align}
where the cross terms vanish by independence, and $\alpha_k^\top C_k\alpha_k=\alpha_k^\top a_k=a_k^\top C_k^{-1}a_k$. Also,
\[
  M_jx_j^{\ind}(\eps)=C_j\alpha_j\eps+\OO(\eps^2)=a_j\eps+\OO(\eps^2),
\]
so
\[
  -\mu_j+2M_jx_j^{\ind}=a_j\eps+\OO(\eps^2),
  \qquad
  \mu_j-3M_jx_j^{\ind}=-2a_j\eps+\OO(\eps^2).
\]
Substituting these expansions into \eqref{eq:Fj-decomp} gives the residual of the full singles-only KKT map at the isolated point:
\begin{equation}
  F_j(x^{\ind}(\eps),\eps)
  = -\Lambda_j a_j\eps^3+\OO(\eps^4).
  \label{eq:residual}
\end{equation}

Set $\Delta_j(\eps):=x_j^{\simu}(\eps)-x_j^{\ind}(\eps)$. By the mean-value theorem,
\[
  0=F(x^{\simu}(\eps),\eps)=F(x^{\ind}(\eps),\eps)+D_xF(\xi(\eps),\eps)\,\Delta(\eps)
\]
for some intermediate point $\xi(\eps)$. Since
\[
  D_xF(\xi(\eps),\eps)=-\diag(C_1,\dots,C_m)+\OO(\eps),
\]
its inverse exists and equals $-\diag(C_1^{-1},\dots,C_m^{-1})+\OO(\eps)$. Combining this with \eqref{eq:residual} yields
\[
  \Delta_j(\eps)=-\Lambda_j C_j^{-1}a_j\,\eps^3+\OO(\eps^4)
  =-\Lambda_j\alpha_j\eps^3+\OO(\eps^4),
\]
which is \eqref{eq:whitrow-main}. The relative form follows by multiplying \eqref{eq:first-order-ray} by $(1-\Lambda_j\eps^2)$.
\end{proof}

\begin{remark}[Interpretation]
The theorem says that the singles-only simultaneous system decouples through second order. The first place where the rest of the portfolio matters is cubic order, and even there the effect is only to shrink each active event block along its own isolated first-order ray. The shrinkage need not be universal: the coefficient $\Lambda_j$ depends on the other events and therefore on $j$.
\end{remark}

\section{How much is lost by forbidding parlays?}
The exact parlay optimizer is the natural benchmark because, by \cref{thm:parlay-main}, it is explicit and eventwise. The singles-only optimizer can only do worse, but the perturbative theorem shows that the loss starts surprisingly late.

\begin{proposition}[Quartic value loss]
\label{prop:quartic-loss}
Let
\[
  V_\sing(\eps):=\mathcal G_{\eps}^{\sing}(x^{\simu}(\eps)),
  \qquad
  V_\parlay(\eps):=\sum_{j=1}^m \mathcal G_{j,\eps}^{\ind}(x_j^{\ind}(\eps)).
\]
Then
\begin{equation}
  0\le V_\parlay(\eps)-V_\sing(\eps)=\OO(\eps^4).
  \label{eq:quartic-loss}
\end{equation}
In particular, excluding parlays changes the growth rate only at quartic order, even though the exact ticket book itself contains all mixed terms.
\end{proposition}

\begin{proof}
By construction, the singles-only feasible set is a restriction of the full parlay feasible set, so $V_\sing(\eps)\le V_\parlay(\eps)$.

For the reverse estimate, evaluate the singles-only objective at the isolated vector $x^{\ind}(\eps)$. Set $u_j:=x_j^{\ind}(\eps)^\top Z_j^{(\eps)}$. Then $u_j=\OO(\eps)$, $\E[u_j]=\OO(\eps^2)$, and the blocks are independent. The exact parlay value at this point is
\[
  \sum_{j=1}^m \E[\log(1+u_j)],
\]
whereas the singles-only value is
\[
  \E\left[\log\left(1+\sum_{j=1}^m u_j\right)\right].
\]
Taylor expansion of the logarithm around $1$ shows that the difference between these two expressions is a sum of mixed expectations involving at least two distinct blocks. By independence, every such mixed term contains at least one factor $\E[u_k]=\OO(\eps^2)$, and all remaining factors are at most $\OO(\eps^2)$ in expectation. Therefore
\begin{equation}
  V_\parlay(\eps)-\mathcal G_{\eps}^{\sing}(x^{\ind}(\eps))=\OO(\eps^4).
  \label{eq:quartic-at-isolated}
\end{equation}
Since $x^{\simu}(\eps)$ maximizes the singles-only objective,
\[
  \mathcal G_{\eps}^{\sing}(x^{\ind}(\eps))
  \le
  \mathcal G_{\eps}^{\sing}(x^{\simu}(\eps))
  =V_\sing(\eps).
\]
Hence
\[
  0\le V_\parlay(\eps)-V_\sing(\eps)
  \le V_\parlay(\eps)-\mathcal G_{\eps}^{\sing}(x^{\ind}(\eps))
  =\OO(\eps^4),
\]
which proves \eqref{eq:quartic-loss}.
\end{proof}

\begin{remark}
The proof only uses the isolated portfolio as a test point. Combining \cref{thm:whitrow} with a Taylor expansion of the singles-only objective shows in fact that
\[
  \mathcal G_{\eps}^{\sing}(x^{\simu}(\eps))
  -\mathcal G_{\eps}^{\sing}(x^{\ind}(\eps))
  =\OO(\eps^6),
\]
so the quartic bound is essentially sharp at the level of the first omitted term.
\end{remark}

\section{Binary check against Thorp}
The cubic shrinkage formula recovers the classical two-binary exact solution. Consider two independent even-money binary bets with return variables
\[
  Z_j\in\{+1,-1\},
  \qquad
  \E[Z_j]=m_j,
\]
so that the isolated Kelly fraction on bet $j$ is exactly $m_j$. In the notation of \cref{thm:whitrow}, this is the one-dimensional case with
\[
  a_j\eps=m_j,\qquad C_j=\E[Z_j^2]=1.
\]
Hence $\alpha_j=C_j^{-1}a_j=a_j$, so the theorem predicts
\[
  f_1^{\simu}=m_1-m_1m_2^2+\OO(\eps^4),
  \qquad
  f_2^{\simu}=m_2-m_2m_1^2+\OO(\eps^4).
\]

In this special binary setting one can go further and solve the simultaneous singles-only problem exactly. Writing $f_1,f_2$ for the two stake fractions, the first-order conditions are
\[
  \E\!\left[\frac{Z_1}{1+f_1Z_1+f_2Z_2}\right]=0,
  \qquad
  \E\!\left[\frac{Z_2}{1+f_1Z_1+f_2Z_2}\right]=0.
\]
A direct calculation gives the unique solution
\begin{equation}
  f_1=\frac{m_1(1-m_2^2)}{1-m_1^2m_2^2},
  \qquad
  f_2=\frac{m_2(1-m_1^2)}{1-m_1^2m_2^2},
  \label{eq:thorp-exact}
\end{equation}
which is the two-bet formula discussed by Thorp \cite{Thorp2006}. Expanding \eqref{eq:thorp-exact} at $(m_1,m_2)=(0,0)$ yields
\[
  f_1=m_1-m_1m_2^2+\OO\!\left((|m_1|+|m_2|)^5\right),
  \qquad
  f_2=m_2-m_2m_1^2+\OO\!\left((|m_1|+|m_2|)^5\right).
\]
Since $m_j=\OO(\eps)$, this becomes
\[
  f_1=m_1-m_1m_2^2+\OO(\eps^5),
  \qquad
  f_2=m_2-m_2m_1^2+\OO(\eps^5).
\]
Thus the general cubic shrinkage law matches the exact binary formula and sharpens to a fifth-order remainder in this symmetric one-dimensional case.

\section{Conclusion}
The exact parlay problem is much simpler than it first appears. Once cash is recognized as an implicit all-state position, the full multibet optimizer is just the outer product of the isolated one-event Kelly strategies. The product structure immediately explains which tickets are active and why no inactive single leg can appear inside an optimal parlay.

That exact factorization also clarifies the singles-only problem. The singles-only wealth is the first-order truncation of the exact product wealth. For small edges, the missing mixed terms cost only quartic growth, and the optimal singles-only stakes are obtained from the isolated Kelly stakes by a cubic, event-specific shrinkage. That is the precise mathematical content behind Whitrow's numerical near-proportionality observation.

\end{document}